\documentclass[a4paper,10pt]{article}
\usepackage{amsmath}
\usepackage{amssymb}
\usepackage{amsfonts}
\usepackage{latexsym}
\usepackage{enumerate}
\usepackage{stmaryrd}

\numberwithin{equation}{section}
\newtheorem{theorem}{Theorem}[section]
\newtheorem{definition}{Definition}[section]

\newtheorem{lemma}{Lemma}[section]
\newtheorem{proposition}{Proposition}[section]
\newtheorem{corollary}{Corollary}[section]
\newtheorem{remark}{Remark}[section]

\begin{document}
\title{\textbf{Certain Ostrowski type inequalities for  generalized $s-$convex functions }}
\author{Muharrem Tomar$^{a}$, Praveen Agarwal\thanks{corresponding author}$^{b}$ and Mohamed Jleli$^{c}$ }
\date{\footnotesize{$^a$Department of Mathematics\\ Faculty of Arts and Sciences, \\
Ordu University, 52200 Ordu, Turkey \\
$^b$Department of Mathematics\\ Anand International College of Engineering\\ Jaipur, India-303012\\
$^c$Department of Mathematics\\ King Saud University, Riyadh-11451\\
Kingdom Saudi Arabia\\
Email: $^a$muharremtomar@gmail.com\\
Email: $^b$goyal.praveen2011@gmail.com\\
Email: $^c$jleli@ksu.edu.sa}}

\maketitle

\begin{abstract}
\noindent In this paper, we first obtain a generalized integral identity for twice local
differentiable functions. Then, using functions whose second derivatives in absolute value
at certain powers are generalized $s-$convex in the second sense, we obtain some new Ostrowski type inequalities.

\end{abstract}
\noindent\textbf{Keywords:} Generalized $s-$convex functions, generalized Hermite-Hadamard inequality, generalized H\"{o}lder inequality.
\section{Introduction and Preliminaries}

Throughout this paper, let $\mathbb{R}$,  $\mathbb{R}^+$, $\mathbb{Q}$,
$\mathbb{Z}$ and $\mathbb{N}$    be the sets of real and positive real numbers, rational numbers,  integers and positive integers,
respectively, and
\begin{equation*}
\mathbb{J}:=\mathbb{R}\setminus \mathbb{Q} \quad \text{and} \quad \mathbb{N}_0:=\mathbb{N}\cup \{0\}.
\end{equation*}

In order to describe the definition of the local fractional
derivative and local fractional integral, recently,
one has introduced to define the following sets (see, \emph{e.g.},  \cite{yang,yang2}; see also \cite{Ch-Se-To}):
For $0<\alpha \leq 1$,
\begin{enumerate}
	\item[(i)] the $\alpha$-type set of integers  $\mathbb{Z}^{\alpha }$ is defined by
	\begin{equation*}
	\mathbb{Z}^{\alpha}:=\{0^\alpha\}\cup\left\{\pm m^\alpha\, :\, m \in \mathbb{N}  \right\};
	\end{equation*}

	\item[(ii)] the $\alpha$-type set of rational numbers $\mathbb{Q}^{\alpha }$ is defined by
	\begin{equation*}
	\mathbb{Q}^{\alpha }:=\left\{q^\alpha\,:\, q \in \mathbb{Q} \right\}=\left\{\left(\frac{m}{n}\right)^\alpha\,:\, m \in \mathbb{Z},\,
	n \in \mathbb{N}\right\};
	\end{equation*}

	\item[(iii)] the $\alpha$-type set of irrational numbers $\mathbb{J}^{\alpha }$ is defined by
	\begin{equation*}
	\mathbb{J}^{\alpha }:=\left\{r^\alpha\,:\, r \in \mathbb{J} \right\}
	=\left\{r^\alpha \ne   \left(\frac{m}{n}\right)^\alpha\,:\, m \in \mathbb{Z},\,
	n \in \mathbb{N}\right\};
	\end{equation*}
	
	\item[(iv)] the $\alpha$-type set of real line numbers $\mathbb{R}^{\alpha}$ is defined by  $\mathbb{R}^{\alpha }:= \mathbb{Q}^{\alpha} \cup  \mathbb{J}^{\alpha}$.
\end{enumerate}

Throughout this paper, whenever the $\alpha$-type set $\mathbb{R}^{\alpha}$ of real line numbers  is involved, the  $\alpha$ is assumed to be tacitly $0<\alpha \leq 1$.

One has also defined two binary operations the addition $+$ and the multiplication $\cdot$ (which is conventionally omitted)  on the $\alpha$-type set $\mathbb{R}^\alpha$ of real line numbers as follows (see, \emph{e.g.},  \cite{yang,yang2}; see also \cite{Ch-Se-To}): For $a^\alpha$, $b^\alpha \in \mathbb{R}^{\alpha}$,
\begin{equation}\label{def-add-multi}
a^\alpha + b^\alpha := (a+b)^\alpha \quad \text{and} \quad  a^\alpha \cdot b^\alpha =  a^\alpha b^\alpha  := (ab)^\alpha.
\end{equation}
Then one finds that
\begin{itemize}
	\item
	$\left(\mathbb{R}^\alpha,\, +\right)$ is a commutative group: For  $a^\alpha$, $b^\alpha$, $c^\alpha \in \mathbb{R}^{\alpha}$,
	\begin{enumerate}
		\item[(A$_1$)] \quad $a^\alpha + b^\alpha \in \mathbb{R}^\alpha$;
		
		\item[(A$_2$)] \quad $a^\alpha + b^\alpha=b^\alpha +a^\alpha$;
		
		\item[(A$_3$)] \quad $a^{\alpha }+\left( b^{\alpha }+c^{\alpha }\right) =\left( a^\alpha+b^\alpha\right)
		+c^{\alpha };$

		\item[(A$_4$)]\quad $0^\alpha$ is the identity for $\left(\mathbb{R}^\alpha,\, +\right)$: For any $a^\alpha \in \mathbb{R}^\alpha$,
		$a^\alpha + 0^\alpha=  0^\alpha +a^\alpha=a^\alpha$;
		
		\item[(A$_5$)] \quad For each $a^\alpha \in \mathbb{R}^\alpha$, $(-a)^\alpha$ is the inverse element of $a^\alpha$ for
		$\left(\mathbb{R}^\alpha,\, +\right)$:
		
		$a^\alpha + (-a)^\alpha =(a+(-a))^\alpha =0^\alpha$;
	\end{enumerate}
	
	\item
	$\left(\mathbb{R}^\alpha \setminus\left\{0^\alpha\right\},\, \cdot\right)$ is a commutative group: For  $a^\alpha$, $b^\alpha$, $c^\alpha \in \mathbb{R}^{\alpha}$,
	\begin{enumerate}
		\item[(M$_1$)] \quad $a^\alpha b^\alpha \in \mathbb{R}^\alpha$;

		\item[(M$_2$)] \quad $a^\alpha \, b^\alpha=b^\alpha \, a^\alpha$;
		
		\item[(M$_3$)] \quad $a^{\alpha}\left( b^{\alpha }c^{\alpha }\right) =\left( a^\alpha\,b^\alpha\right)c^{\alpha };$

		\item[(M$_4$)]\quad $1^\alpha$ is the identity for $\left(\mathbb{R}^\alpha,\, \cdot\right)$: For any $a^\alpha \in \mathbb{R}^\alpha$,
		$a^\alpha  1^\alpha=  1^\alpha a^\alpha=a^\alpha$;
		
		\item[(M$_5$)] \quad For each $a^\alpha \in \mathbb{R}^\alpha \setminus\left\{0^\alpha\right\}$, $(1/a)^\alpha$ is the inverse element of $a^\alpha$ for
		$\left(\mathbb{R}^\alpha,\, \cdot\right)$:
		
		$a^\alpha  (1/a)^\alpha =(a(1/a))^\alpha =1^\alpha$;
	\end{enumerate}

	\item Distributive law holds:
	$a^{\alpha }\left( b^{\alpha }+c^{\alpha }\right)
	=a^{\alpha }b^{\alpha}+a^{\alpha }c^{\alpha }.$
	
\end{itemize}

\vskip 3mm Furthermore we observe some additional properties for $\left(\mathbb{R}^\alpha,\, +,\,\cdot\right)$
which are stated in the following proposition (see \cite{Ch-Se-To}).

\vskip 3mm

\begin{proposition} \label{prop-1} Each of the following statements holds true:
	\begin{enumerate}
		\item[\rm (i)] Like the usual real number system $(\mathbb{R},\,+,\,\cdot)$,
		$\left(\mathbb{R}^\alpha,\, +,\,\cdot\right)$ is a field;

		\item[\rm (ii)] The additive identity $0^\alpha$ and the multiplicative identity $1^\alpha$
		are unique, respectively;
		
		\item[\rm (iii)] The additive inverse element and the multiplicative inverse element
		are unique, respectively;
		
		\item[\rm (iv)] For each $a^\alpha \in \mathbb{R}^\alpha$, its inverse element $(-a)^\alpha$ may be written as $-a^\alpha$;
		for each  $b^\alpha \in \mathbb{R}^\alpha \setminus\left\{0^\alpha\right\}$, its inverse element $(1/b)^\alpha$
		may be written as $1^\alpha/b^\alpha$ but not as $1/b^\alpha$;
		
		\item[\rm (v)] If the order $<$ is defined on $\left(\mathbb{R}^\alpha,\,+,\,\cdot\right)$ as follows:
		$a^\alpha <b^\alpha$ in $\mathbb{R}^\alpha$ if and only if $a<b$ in $\mathbb{R}$,
		then $\left(\mathbb{R}^\alpha,\,+,\,\cdot,\, <\right)$ is an ordered field like $\left(\mathbb{R},\,+,\,\cdot,\, <\right)$.
	\end{enumerate}
	
\end{proposition}

\vskip 3mm
In order to introduce the local fractional calculus on $\mathbb{R}^{\alpha }$,
we begin with the concept of the local fractional continuity as in Definition \ref{l-conti}.
\vskip 3mm

\begin{definition}\label{l-conti}
	A non-differentiable function $f:\mathbb{R} \rightarrow \mathbb{R}^{\alpha }$,
	$x \mapsto f(x)$,  is called to be local fractional continuous at $x_{0}$
	if for any $\varepsilon \in \mathbb{R}^+$, there exists $\delta \in \mathbb{R}^+$ such that
	\begin{eqnarray*}
		\left \vert f(x)-f(x_{0})\right \vert <\varepsilon ^{\alpha }
	\end{eqnarray*}
\end{definition}
holds for $\left \vert x-x_{0}\right \vert <\delta$.
If a function $f$ is local continuous on the interval $\left(a,b\right)$,  we denote $f \in C_{\alpha }(a,b)$.

\vskip 3mm
Among several attempts to have defined
local fractional derivative and local fractional
integral (see \cite[Section 2.1]{yang-11}),
we choose to recall the following definitions of local
fractional calculus (see, \emph{e.g.},  \cite{sarikaya8, yang-11, yang}):
\vskip 3mm

\begin{definition} \label{l-derivative}
	The local fractional derivative of $f(x)$ of order $\alpha $ at
	$x=x_{0}$ is defined by
	\begin{eqnarray*}
		f^{(\alpha )}(x_{0})={}_{x_0}D_x^\alpha f(x)   = \left. \frac{d^{\alpha }f(x)}{dx^{\alpha }}
		\right \vert_{x=x_{0}}=\lim_{x\rightarrow x_{0}}\frac{\Delta ^{\alpha }\left(
			f(x)-f(x_{0})\right) }{\left( x-x_{0}\right) ^{\alpha }},
	\end{eqnarray*}
	where $\displaystyle \Delta ^{\alpha }\left( f(x)-f(x_{0})\right) =\Gamma
	(\alpha +1)\left( f(x)-f(x_{0})\right)$ and $\Gamma$ is the familiar Gamma function
	$($see, e.g., \cite[Section 1.1]{Sr-Ch-12}$)$.

	Let  $f^{(\alpha )}(x)=D_x^\alpha f(x)$.  If there exists $ \displaystyle f^{(k+1)\alpha }(x)
	=\displaystyle\overset{k+1\, {\rm times}}{\overbrace{D_{x}^{\alpha}\dots D_{x}^{\alpha }}}f(x) $ for any $x\in I\subseteq \mathbb{R}$,  then
	we denote $f\in D_{(k+1)\alpha }(I)$ $\left(k \in \mathbb{N}_0\right)$.
	
\end{definition}

\vskip 3mm

\begin{definition} \label{l-integral}
	Let $f \in C_{\alpha }\left[ a,b\right]$. Also let $P=\left\{t_0,\,\ldots,\,t_N \right\},$ $(N \in \mathbb{N})$ be
	a partition of the interval $[a,b]$ which satisfies
	$a=t_{0}<t_{1}<\cdots<t_{N-1}<t_{N}=b$.
	Further, for this partition $P$,  let  $\Delta t:=\max\limits_{0\leq j \leq N-1}\Delta t_{j}$
	where $\Delta t_{j}:=t_{j+1}-t_{j}$ $\left\{j=0,\, \ldots,\, N-1\right\}$.
	Then the local fractional integral of $f$ on the interval $[a,b]$ of order $\alpha$ $($denoted by ${}_{a}I_{b}^{(\alpha) }f)$  is defined by
	\begin{equation} \label{def-lfi}
	{}_{a}I_{b}^{(\alpha) }f(t) =\frac{1}{\Gamma (\alpha +1)}\int_{a}^{b}f(t)(dt)^{\alpha }
	:=\frac{1}{\Gamma (\alpha +1)}\lim_{\Delta
		t\rightarrow 0}\sum \limits_{j=0}^{N-1}f(t_{j})(\Delta t_{j})^{\alpha },
	\end{equation}
	provided the limit exists $($in fact, this limit exists if $f \in C_{\alpha }\left[ a,b\right]$$)$.

	Here, it follows that $_{a}I_{b}^{(\alpha) }f=0$ if $a=b$ and
	$_{a}I_{b}^{(\alpha) }f=-_{b}I_{a}^{(\alpha) }f$ if $a<b.$
	
	If  $_{a}I_{x}^{(\alpha) }g$  exists  for any $x\in \left[ a,b\right]$ and a function $g:[a,\,b] \rightarrow \mathbb{R}^\alpha$,
	then we denote
	$g \in I_{x}^{(\alpha) }\left[ a,b\right]$.
\end{definition}

\vskip 3mm

We give  some of the features related to the local fractional calculus
that will be required for our main results (see \cite{yang}).

\vskip 3mm
\begin{lemma}
	\label{belirli-integral}
	The following identities hold true:
	\begin{enumerate}[$(a)$]
		\item
		$(1\alpha-$local fractional derivative of $x^{k\alpha}$$)$ \\
		\begin{eqnarray*}
			\displaystyle\frac{d^{\alpha }x^{k\alpha }}{dx^{\alpha }}=\frac{\Gamma (1+k\alpha )}{%
				\Gamma (1+\left( k-1\right) \alpha )}x^{\left( k-1\right) \alpha }.
		\end{eqnarray*}
		
		\item
		$($Local fractional integration is anti-differentiation$)$  \\
		Suppose that
		$f(x)=g^{(\alpha )}(x)\in C_{\alpha }\left[ a,b\right]$. Then we have
		\begin{eqnarray*}
			_{a}I_{b}^{\alpha }f(x)=g(b)-g(a).
		\end{eqnarray*}
		\item
		$($Local fractional integration by parts$)$ \\
		Suppose that $f(x),g(x)\in D_{\alpha }\left[ a,b\right] $
		and $f^{(\alpha )}(x),$ $g^{(\alpha )}(x)\in C_{\alpha }\left[ a,b\right]$.
		Then we have
		\begin{eqnarray*}
			_{a}I_{b}^{\alpha }f(x)g^{(\alpha )}(x)=\left. f(x)g(x)\right \vert
			_{a}^{b}-_{a}I_{b}^{\alpha }f^{(\alpha )}(x)g(x).
		\end{eqnarray*}

		\item $($Local fractional definite integrals of $x^{k\alpha})$
		\begin{eqnarray*}
			\displaystyle\frac{1}{\Gamma (1+\alpha )}\int_{a}^{b}x^{k\alpha }(dx)^{\alpha }=
			\frac{\Gamma (1+k\alpha )}{\Gamma (1+\left( k+1\right) \alpha )}\left(
			b^{\left( k+1\right) \alpha }-a^{\left( k+1\right) \alpha }\right)
			\quad (k \in \mathbb{R}).
		\end{eqnarray*}
	\end{enumerate}
\end{lemma}

\vskip 3mm
For further details on local fractional calculus,
one may refer to \cite{yang-11}-\cite{yang5}.

\vskip 3mm
Let $I$ be an interval in $\mathbb{R}$.
A function $f:I \rightarrow \mathbb{R}^\alpha$
is said to be convex on $I$ if
\begin{eqnarray}
f\left( tx+\left( 1-t\right) y\right) \leq tf\left( x\right) +\left(
1-t\right) f\left( y\right)
\end{eqnarray}
holds for every $x,\,y\in I\ $  and $t\in \left[ 0,1\right] $.

If a function $f:I \subset \mathbb{R}  \rightarrow \mathbb{R}$ ($I$ an interval)
is convex on $I$, then, for $a,\, b\in I$ with $a<b$, we have
\begin{equation} \label{H.H}
f\left( \frac{a+b}{2}\right) \leq \frac{1}{b-a}\int_{a}^{b}f\left( x\right)
dx\leq \frac{f\left( a\right) +f\left( b\right) }{2},
\end{equation}
which is known as the Hermite-Hadamard inequality.


\vskip 3mm
Mo \emph{et al}. \cite{mo} introduced the following generalized convex function.

\begin{definition}\label{gcf}
	Let $f:I \subset \mathbb{R}  \rightarrow \mathbb{R}^\alpha$ ($I$ an interval)
	be a function.
	If, for any
	$x_{1},\, x_{2}\in I$ and $\lambda \in \left[ 0,1\right]$,  the following
	inequality
	\begin{eqnarray*}
		f\left( \lambda x_{1}+(1-\lambda) x_{2}\right)
		\leq \lambda^\alpha f(x_1)+{(1-\lambda)}^\alpha f(x_2)
	\end{eqnarray*}
	holds, then $f$ is called a generalized convex function on $I.$
\end{definition}

Here are two basic examples of generalized convex functions:

\begin{enumerate}
	\item[(1)] $f\left( x\right) =x^{\alpha p}$ $( p>1)$;
	
	\item[(2)] $g\left( x\right) =E_{\alpha }\left( x^{\alpha }\right)$ $( x\in
	\mathbb{R})$, where $E_{\alpha }\left( x^{\alpha }\right) :=\sum_{k=0}^{\infty }\frac{%
		x^{\alpha k}}{\Gamma (1+k\alpha )}$ is the Mittag-Leffer function.
\end{enumerate}

\vskip 3mm
Recently the fractal theory has received a significant attention
(see, \emph{e.g.}, \cite{budak,sarikaya8, mo2,mo4,sarikaya5,sarikaya6,sarikaya9,tomar}).
Mo \emph{et al}.  \cite{mo}   proved the following
analogue of  the Hermite-Hadamard inequality \eqref{H.H} for generalized convex functions:
Let $f\left( x\right) \in I_{x}^{\alpha }\left[ a,b\right]$ be a
generalized convex function on $\left[ a,b\right]$ with $a<b$. Then we have
\begin{eqnarray}
\label{GHH}
f\left( \frac{a+b}{2}\right) \leq \frac{\Gamma (1+\alpha )}{\left(
	b-a\right) ^{\alpha }}\ _{a}I_{b}^{\alpha }f(x)\leq \frac{f\left( a\right)
	+f\left( b\right) }{2^{\alpha }}\text{.}
\end{eqnarray}

\begin{remark} \rm
	The double inequality \eqref{GHH} is known in the literature as generalized Hermite-Hadamard integral inequality for generalized convex functions.
	Some of the classical inequalities for means can be derived
	from \eqref{GHH} with  appropriate selections of the mapping $f$.
	Both inequalities in \eqref{H.H} and \eqref{GHH} hold in the reverse direction if $f$ is concave and generalized concave, respectively.
	For some more results which generalize, improve and extend the
	inequalities \eqref{GHH}, one may refer  to the recent papers
	\cite{budak,mo2,mo4}, \cite{sarikaya5}-\cite{sarikaya9} and references therein.
\end{remark}

\vskip 3mm

An analogue in the fractal set $\mathbb{R}^\alpha$ of the classical
H\"{o}lder's  inequality has been established by Yang \cite{yang}, which is  asserted by the
following lemma.

\begin{lemma} \label{l3}
	Let $f,g\in C_{\alpha }\left[ a,b\right],$
	with $\frac{1}{p}+\frac{1}{q}=1$   $(p,\,q>1)$.
	Then we have
	\begin{eqnarray}
	&&\frac{1}{\Gamma (\alpha +1)}\int_{a}^{b}\left\vert
	f(x)g(x)\right\vert (dx)^{\alpha } \\
	&&\leq \left( \dfrac{1}{\Gamma (\alpha +1)}
	\int_{a}^{b}\left\vert f(x)\right\vert ^{p}(dx)^{\alpha }\right) ^{
		\frac{1}{p}}\left( \dfrac{1}{\Gamma (\alpha +1)}\int_{a}^{b}\left
	\vert g(x)\right\vert ^{q}(dx)^{\alpha }\right) ^{\frac{1}{q}}. \nonumber
	\end{eqnarray}
\end{lemma}

\vskip 3mm

\begin{theorem}[Generalized Ostrowski inequality]
Let $I\subseteq
\mathbb{R}
$ be an interval,
$f:I^{0}\subseteq\mathbb{R}\rightarrow\mathbb{R}^{\alpha }$
($I^{0}$ is the interior of $I$) such that $f\in D_{\alpha
}(I^{0})$ and $f^{(\alpha )}\in C_{\alpha }\left[ a,b\right] $ for $a,b\in
I^{0}$ with $a<b$ Then. for all $x\in \left[ a,b\right] ,$ we have the
identity%
\begin{equation}
\left \vert f(x)-\frac{\Gamma \left( 1+\alpha \right) }{\left( b-a\right)
^{\alpha }}\text{ }_{a}I_{b}^{\alpha }f(t)\right \vert \leq 2^{\alpha }\frac{%
\Gamma \left( 1+\alpha \right) }{\Gamma \left( 1+2\alpha \right) }\left[
\frac{1}{4^{\alpha }}+\left( \frac{x-\frac{a+b}{2}}{b-a}\right) ^{2\alpha }%
\right] \left( b-a\right) ^{\alpha }\left \Vert f^{\left( \alpha \right)
}\right \Vert _{\infty }.  \label{E2}
\end{equation}
\end{theorem}

In \cite{mo4}, Mo and Sui established the following Hermite-Hadamard
inequality for generalized $s-$convex functions on real linear fractal set $
\mathbb{R}^{\alpha }$ $(0<\alpha <1):$

\begin{theorem}
Suppose that
$f:\mathbb{R}_{+}\rightarrow\mathbb{R}^\alpha
$ is a generalized $s-$convex function in the second sense, where $%
s\in (0,1)$. Let $a,b\in \lbrack 0,\infty )$, $a<b$. If $f\in C_{\alpha }%
\left[ a,b\right] $, then the following inequalities hold:%
\begin{equation}
\frac{2^{\left( s-1\right) \alpha }}{\Gamma \left( 1+\alpha \right) }f\left(
\frac{a+b}{2}\right) \leq \frac{\text{ }_{a}I_{b}^{\alpha }f(t)}{\left(
b-a\right) ^{\alpha }}\leq \frac{\Gamma \left( 1+s\alpha \right) }{\Gamma
\left( 1+\left( s+1\right) \alpha \right) }\left( f(a)+f(b)\right) .
\label{E3}
\end{equation}
If $f$ is a generalized $s-$concave, then we have the reverse inequality.
\end{theorem}

In this section, we first obtain a generalized integral identity for functions  twice local
differentiable functions. Then, we use this identity to obtain our results and using functions whose second derivatives in absolute value at certain powers are generalized $s-$convex, obtain some new Ostrowski type inequalities for functions whose local fractional derivatives are generalized $s-$convex in the second sense.

\section{Main Results}
\begin{lemma}
\label{mean-lemma}
Let $I\subseteq\mathbb{R}$ be an interval,
$f:I^\circ\subseteq\mathbb{R\rightarrow\mathbb{R}^\alpha}$ $(I^\circ$ is the interior of $I)$
such that $f^{(\alpha)}\in D_\alpha(I^\circ)$ and $f^{(2\alpha)}\in C_\alpha[a,b]$ for $a,b\in I^\circ$
with $a<b$. Then, for all $x\in[a,b]$ we have the identity,
\begin{eqnarray}
\label{m-l}
&&\frac{1}{(b-a)^\alpha}\ _aI_b^\alpha f(t) -\frac{f(x)}{\Gamma(1+\alpha)}+\frac{(2x-a-b)^\alpha f^{(\alpha)}(x)}{\Gamma(1+2\alpha)}\\
&=&\frac{(x-a)^{3\alpha}}{\Gamma(1+\alpha)\Gamma(1+2\alpha)(b-a)^\alpha}\int_0^1t^{2\alpha}f^{(2\alpha)}(tx+(1-t)a)(dt)^\alpha \nonumber\\
&&+\frac{(b-x)^{3\alpha}}{\Gamma(1+\alpha)\Gamma(1+2\alpha)(b-a)^\alpha}\int_0^1t^{2\alpha}f^{(2\alpha)}(tx+(1-t)b)(dt)^\alpha. \nonumber
\end{eqnarray}
\end{lemma}
\begin{proof}
Using the local fractional integration by parts, we have
\begin{eqnarray}
\label{ml}
&&\frac{1}{\Gamma(1+\alpha)}\int_0^1t^{2\alpha}f^{(2\alpha)}(tx+(1-t)a)(dt)^\alpha \\
&=&\frac{t^{2\alpha}}{(x-a)^\alpha}f^{(\alpha)}(tx+(1-t)a)\Big|_0^1 \nonumber\\
&&-\frac{\Gamma(1+2\alpha)}{\Gamma(1+\alpha)(x-a^\alpha)}\frac{1}{\Gamma(1+\alpha)}\int_0^1t^\alpha f^{(\alpha)}(tx+(1-t)a) \nonumber\\
&=&\frac{f^{(\alpha)}(x)}{(x-a)^\alpha}
-\frac{\Gamma(1+2\alpha)}{(x-a)^\alpha\Gamma(1+\alpha)}\Bigg[\frac{t^\alpha}{(x-a^\alpha)}f(tx+(1-t)a)\Big|_0^1 \nonumber \\
&&-\frac{\Gamma(1+\alpha)}{(x-a)^\alpha}\frac{1}{\Gamma(1+\alpha)}\int_0^1f(tx+(1-t)a)(dt)^\alpha \Bigg] \nonumber \\
&=&\frac{f^{(\alpha)}(x)}{(x-a)^\alpha}-\frac{\Gamma(1+2\alpha)f(x)}{\Gamma(1+\alpha)(x-a)^{2\alpha}}\nonumber \\
&&+\frac{\Gamma(1+2\alpha)}{(x-a)^{2\alpha}}\frac{1}{\Gamma(1+\alpha)}\int_0^1f(tx+(1-t)a)(dt)^\alpha \nonumber
\end{eqnarray}
By using the change of the variable $u=tx+(1-t)a$ for $t\in[0,1]$ and by multiplying
the both sides of \eqref{ml}  by $\frac{(x-a)^{3\alpha}}{\Gamma(1+2\alpha)(b-a)^\alpha}$, we obtain

\begin{eqnarray}
\label{ml1}
&&\frac{(x-a)^{3\alpha}}{\Gamma(1+\alpha)\Gamma(1+2\alpha)(b-a)^\alpha}\int_0^1t^{2\alpha}f^{(2\alpha)}(tx+(1-t)a)(dt)^\alpha \nonumber\\
&=&\frac{(x-a)^{2\alpha}f^{(\alpha)}(x)}{\Gamma(1+2\alpha)(b-a)^\alpha}-\frac{(x-a)^\alpha f(x)}{\Gamma(1+\alpha)(b-a)^\alpha}\nonumber \\
&&+\frac{1}{(b-a)^\alpha}\ _aI_x^\alpha f(t).
\end{eqnarray}

Analogously, we also have the following equality:
\begin{eqnarray}
\label{ml2}
&&\frac{(b-x)^{3\alpha}}{\Gamma(1+\alpha)\Gamma(1+2\alpha)(b-a)^\alpha}\int_0^1t^{2\alpha}f^{(2\alpha)}(tx+(1-t)b)(dt)^\alpha \nonumber\\
&=&-\frac{(b-x)^{2\alpha}f^{(\alpha)}(x)}{\Gamma(1+2\alpha)(b-a)^\alpha}-\frac{(b-x)^\alpha f(x)}{\Gamma(1+\alpha)(b-a)^\alpha}\nonumber \\
&&+\frac{1}{(b-a)^\alpha}\ _xI_b^\alpha f(t).
\end{eqnarray}
So, adding \eqref{ml1} and \eqref{ml2}, we get desired inequality \eqref{m-l}.This completes the proof of the lemma.
\end{proof}
\begin{theorem}
\label{Theorem-1}
 Suppose that the assumptions of Lemma \ref{mean-lemma}  are satisfied. If $|f^{(2\alpha)}|$
 is generalized $s-$convex in the second sense where $s\in(0,1)$, then \label{eq1}
\begin{eqnarray}
\label{t-1-1}
   &&\left|\frac{1}{(b-a)^\alpha}\ _aI_b^\alpha f(t)-\frac{f(x)}{\Gamma(1+\alpha)}
   +\frac{(2x-a-b)^{\alpha}f^{(\alpha)}(x)}{\Gamma(1+2\alpha)}\right| \nonumber\\
   &\leq&\frac{(x-a)^{3\alpha}}{\Gamma(1+2\alpha)(b-a)^{\alpha}}\left(M(s,\alpha)|f^{(2\alpha)}(x)|+N(s,\alpha)|f^{(2\alpha)}(a)|\right)  \nonumber\\
   &+&\frac{(b-x)^{3\alpha}}{\Gamma(1+2\alpha)(b-a)^{\alpha}}\left(M(s,\alpha)|f^{(2\alpha)}(x)|+N(s,\alpha)|f^{(2\alpha)}(b)|\right)
\end{eqnarray}
where
$M(s,\alpha)\displaystyle=\frac{\Gamma(1+(s+2)\alpha)}{\Gamma(1+(s+3)\alpha)}$  \:\:\:\: and \:\:\:\: $N(s,\alpha)\displaystyle=\frac{\Gamma(1+s\alpha)}{\Gamma(1+(s+1)\alpha)}
-2^\alpha\frac{\Gamma(1+(s+1)\alpha)}{\Gamma(1+(s+2)\alpha)} +\frac{\Gamma(1+(s+2)\alpha)}{\Gamma(1+(s+3)\alpha)}$.
\end{theorem}
\begin{proof} Taking modulus in Lemma \ref{mean-lemma} and generalized  $s-$convexity in the second sense of  $|f^{(2\alpha)}|$, we have
\begin{eqnarray}
\label{eq2}
   &&\left\vert\frac{1}{(b-a)^\alpha}\ _aI_b^\alpha f(t)-\frac{f(x)}{\Gamma(1+\alpha)}+\frac{(2x-a-b)^{\alpha}f^{(\alpha)}(x)}{\Gamma(1+2\alpha)}\right\vert \\
   &\leq&\frac{(x-a)^{3\alpha}}{\Gamma(1+\alpha)\Gamma(1+2\alpha)(b-a)^{\alpha}}\bigg\{\int_{0}^{1}
   t^{2\alpha}\Big|f^{(2\alpha)}(tx+(1-t)a)\Big|(dt)^{\alpha} \nonumber \\
   &+&\frac{(b-x)^{3\alpha}}{\Gamma(1+\alpha)\Gamma(1+2\alpha(b-a)^{\alpha}}\int_{0}^{1}
   t^{2\alpha}\Big|f^{(2\alpha)}(tx+(1-t)b)\Big|(dt)^{\alpha}\bigg\}  \nonumber  \\
   &\leq&\frac{(x-a)^{3\alpha}}{\Gamma(1+2\alpha)(b-a)^{\alpha}}\left[\frac{1}{\Gamma(1+\alpha)}
   \int_{0}^{1}\Big(t^{2\alpha}\Big[t^{s\alpha}\left|f^{(2\alpha)}(x)\right|+(1-t)^{s\alpha}
   \Big|f^{(2\alpha)}(a)\Big|\Big]\Big)(dt)^{\alpha}\right] \nonumber   \\
   &+&\frac{(b-x)^{3\alpha}}{\Gamma(1+2\alpha)(b-a)^{\alpha}}\left[\frac{1}{\Gamma(1+\alpha)}
   \int_{0}^{1}\left(t^{2\alpha}\Big[t^{s\alpha}\left|f^{(2\alpha)}(x)\right|+(1-t)^{s\alpha}
   \Big|f^{(2\alpha)}(b)\Big|\Big)\Big]\right)(dt)^{\alpha}\right] \nonumber
\end{eqnarray}
Using Lemma(\ref{belirli-integral}), we also have
\begin{eqnarray} \label{eq3}
\frac{1}{\Gamma(1+\alpha)}\int_{0}^{1}t^{2\alpha}t^{s\alpha}=\frac{\Gamma(1+(s+2)\alpha)}{\Gamma(1+(s+3)\alpha)}
\end{eqnarray}
and
\begin{eqnarray} \label{eq4}
\frac{1}{\Gamma(1+\alpha)}\int_{0}^{1}t^{2\alpha}(1-t)^{s\alpha}=\frac{\Gamma(1+s\alpha)}{\Gamma(1+(s+1)\alpha)}
-2^\alpha\frac{\Gamma(1+(s+1)\alpha)}{\Gamma(1+(s+2)\alpha)}+\frac{\Gamma(1+(s+2)\alpha)}{\Gamma(1+(s+3)\alpha)}.\nonumber \\
\end{eqnarray}
If subsatýtute equalities \eqref{eq3}  and \eqref{eq4}  in \eqref{eq2}, we get desired  inequality \eqref{t-1-1}. So, the proof
is complete.
\end{proof}
\begin{corollary}\label{corollary_1}
In Theorem \ref{Theorem-1}, if we choose $x =\frac{a+b}{2}$
and use the $s-$convexity of $|f^{(2\alpha)}|$, we obtain
\begin{eqnarray}
\label{t-1-1}
   &&\left\vert\frac{1}{(b-a)^\alpha}\ _aI_b^\alpha f(t)-\frac{f(\frac{a+b}{2})}{\Gamma(1+\alpha)}\right\vert \\
   &\leq&\frac{(b-a)^{2\alpha}}{8^\alpha\Gamma(1+2\alpha)}\left(M(s,\alpha)|f^{(2\alpha)}(\frac{a+b}{2})|+N(s,\alpha)|f^{(2\alpha)}(a)|\right)  \nonumber\\
   &&+\frac{(b-a)^{2\alpha}}{8^\alpha\Gamma(1+2\alpha)}\left(M(s,\alpha)|f^{(2\alpha)}(\frac{a+b}{2})|+N(s,\alpha)|f^{(2\alpha)}(b)|\right)\nonumber\\
   &\leq&\frac{(b-a)^{2\alpha}}{8^\alpha\Gamma(1+2\alpha)}\left(\frac{2^{(1-s)\alpha}M(s,\alpha)+2^{s\alpha}N(s,\alpha)}{2^{s\alpha}}\right)
   \left(\left\vert f^{(2\alpha)}(a)\right\vert+\left\vert f^{(2\alpha)}(b)\right\vert\right).\nonumber
\end{eqnarray}
\end{corollary}

\begin{corollary}
\label{corollary-1}
Taking $\Theta:=\left\vert f^{(2\alpha)}(x)\right\vert_\infty$ in Theorem \ref{Theorem-1}, we get
\begin{eqnarray}
\label{t-1-2}
   &&\left|\frac{1}{(b-a)^\alpha}\ _aI_b^\alpha f(t)-\frac{f(x)}{\Gamma(1+\alpha)}
   +\frac{(2x-a-b)^{\alpha}f^{(\alpha)}(x)}{\Gamma(1+2\alpha)}\right| \nonumber\\
   &\leq&\frac{3^\alpha\Theta\left(M(s,\alpha)+N(s,\alpha)\right)}{\Gamma(1+2\alpha)}
   \left[\frac{\left(b-a\right)^{2\alpha}}{12^\alpha}+\left(x-\frac{a+b}{2}\right)^{2\alpha}\right].
\end{eqnarray}
\end{corollary}
\begin{corollary}
\label{corollary-2}
If we take $x=\frac{a+b}{2}$ in Corollary \ref{corollary-1}, we get
\begin{eqnarray}
\label{t-1-2}
   &&\left\vert\frac{1}{(b-a)^\alpha}\ _aI_b^\alpha f(t)-\frac{f\left(\frac{a+b}{2}\right)}{\Gamma(1+2\alpha)}
   \right\vert\leq\frac{\Theta\left(M(s,\alpha)+N(s,\alpha)\right)\left(b-a\right)^{2\alpha}}{4^\alpha\Gamma(1+2\alpha)}. \nonumber \\
\end{eqnarray}
\end{corollary}
\begin{theorem}
\label{Theorem-2}
Suppose that the assumptions of Lemma \ref{mean-lemma}  are satisfied.
If $|f^{(2\alpha)}|^q$  is generalized $s$-convex in the second sense where $s\in (0,1)$, then \label{eq1}
\begin{eqnarray}
\label{S}
&&\left\vert\frac{1}{(b-a)^\alpha}\ _aI_b^\alpha f(t)-\frac{f(x)}{\Gamma(1+\alpha)}
   +\frac{(2x-a-b)^{\alpha}f^{(\alpha)}(x)}{\Gamma(1+2\alpha)}\right\vert \nonumber \\
   &&\leq\left(\frac{\Gamma(1+2p\alpha)}{\Gamma(1+(2p+1)\alpha)}\right)^\frac{1}{p}
            \left(\frac{\Gamma(1+s\alpha)}{\Gamma(1+(s+1)\alpha)}\right)^\frac{1}{q}
                    \frac{1}{\Gamma(1+2\alpha)(b-a)^\alpha} \nonumber \\
                    &&\:\:\:\:\:\:\:\:\:\:\times \Bigg[(x-a)^{3\alpha}\left(\left\vert f^{(2\alpha)}(x)
                        \right\vert^q+ \left\vert f^{(2\alpha)}(a)\right\vert^q\right)^\frac{1}{q} \nonumber\\
                    && \:\:\:\:\:\:\:\:\:\:\:\:\:\:\:\:\:\:\:\:\:\:\:\:\:\:\:\:\:\:
                        +(b-x)^{3\alpha}\left(\left\vert f^{(2\alpha)}(x)\right\vert^q
                             + \left\vert f^{(2\alpha)}(b)\right\vert^q\right)^\frac{1}{q}\Bigg]
\end{eqnarray}
where $\frac{1}{p}+\frac{1}{q}=1$.
\end{theorem}
\begin{proof}
Taking modulus Lemma \ref{mean-lemma} and by generalized H\"{o}lder inequality, we have
\begin{eqnarray}
\label{S-1}
   &&\left\vert\frac{1}{(b-a)^\alpha}\ _aI_b^\alpha f(t)-\frac{f(x)}{\Gamma(1+\alpha)}
   +\frac{(2x-a-b)^{\alpha}f^{(\alpha)}(x)}{\Gamma(1+2\alpha)}\right\vert \nonumber \\
   &\leq&\frac{(x-a)^{3\alpha}}{\Gamma(1+\alpha)\Gamma(1+2\alpha)(b-a)^{\alpha}}\bigg\{\int_{0}^{1}
   t^{2\alpha}\left\vert f^{(2\alpha)}(tx+(1-t)a)\right\vert(dt)^{\alpha} \nonumber \\
   &&+\frac{(b-x)^{3\alpha}}{\Gamma(1+\alpha)\Gamma(1+2\alpha(b-a)^{\alpha}}\int_{0}^{1}
   t^{2\alpha}\Big|f^{(2\alpha)}(tx+(1-t)b)\Big|(dt)^{\alpha}\bigg\}  \nonumber  \\
   &\leq&\frac{(x-a)^{3\alpha}}{\Gamma(1+2\alpha)(b-a)^{\alpha}}
   \left(\frac{1}{\Gamma(1+\alpha)}\int_0^1t^{2p\alpha}(dt)^\alpha\right)^\frac{1}{p} \nonumber \\
   &&\:\:\:\:\:\:\:\:\:\:\:\:\:\:\:\:\times\left(\frac{1}{\Gamma(1+\alpha)}\int_0^1\left\vert f^{(2\alpha)}(tx+(1-t)a)\right\vert^q(dt)^\alpha\right)^\frac{1}{q} \nonumber \nonumber \\
   &&+\frac{(b-x)^{3\alpha}}{\Gamma(1+2\alpha)(b-a)^{\alpha}}
   \left(\frac{1}{\Gamma(1+\alpha)}\int_0^1t^{2p\alpha}(dt)^\alpha\right)^\frac{1}{p} \nonumber \\
   &&\:\:\:\:\:\:\:\:\:\:\:\:\:\:\:\:\times\left(\frac{1}{\Gamma(1+\alpha)}\int_0^1\left\vert f^{(2\alpha)}(tx+(1-t)b)\right\vert^q(dt)^\alpha\right)^\frac{1}{q}.
\end{eqnarray}
Since $\left\vert f^{(2\alpha)}\right\vert^q$ is generalized $s-$convex in the second sense
and from generalized Hermite-Hadamard inequality for $s-$convex functions in the second sense,
we have
\begin{eqnarray}
\label{S__2}
\int_0^1 \left\vert f^{(2\alpha)}(tx+(1-t)a)\right\vert^q(dt)^\alpha
=\frac{1}{(x-a)^\alpha}\int_a^x\left\vert f^{(2\alpha)}(u)\right\vert^q(du)^\alpha\nonumber \\
\leq\frac{\Gamma(1+s\alpha)}{\Gamma(1+(s+1)\alpha)}
\left(\left\vert f^{(2\alpha)}(x)\right\vert^q+\left\vert f^{(2\alpha)}(a)\right\vert^q\right)
\end{eqnarray}
and similarly
\begin{eqnarray}
\label{S__3}
&&\int_0^1 \left\vert f^{(2\alpha)}(tx+(1-t)a)\right\vert^q(dt)^\alpha\nonumber \\
&&\leq\frac{\Gamma(1+s\alpha)}{\Gamma(1+(s+1)\alpha)}\left(\left\vert f^{(2\alpha)}(x)\right\vert^q+\left\vert f^{(2\alpha)}(b)\right\vert^q\right).
\end{eqnarray}
From Lemma \ref{belirli-integral}, we also have
\begin{eqnarray}
\label{S__4}
\frac{1}{\Gamma(1+\alpha)}\int_0^1t^{2p\alpha}(dt)^\alpha
=\frac{\Gamma(1+2p\alpha)}{\Gamma(1+(2p+1)\alpha)}.
\end{eqnarray}
Now, if we substitute inequalities \eqref{S__2}, \eqref{S__3} and equality \eqref{S__4} in \eqref{S}, we obtain
\begin{eqnarray}
&&\left\vert\frac{1}{(b-a)^\alpha}\ _aI_b^\alpha f(t)-\frac{f(x)}{\Gamma(1+\alpha)}
   +\frac{(2x-a-b)^{\alpha}f^{\alpha}(x)}{\Gamma(1+2\alpha)}\right\vert \nonumber \\
   &&\leq\left(\frac{\Gamma(1+2p\alpha)}{\Gamma(1+(2p+1)\alpha)}\right)^\frac{1}{p}
            \left(\frac{\Gamma(1+s\alpha)}{\Gamma(1+(s+1)\alpha)}\right)^\frac{1}{q}
                    \frac{1}{\Gamma(1+2\alpha)(b-a)^\alpha} \nonumber \\
                    &&\:\:\:\:\:\:\:\:\:\:\times \Bigg[(x-a)^{3\alpha}\left(\left\vert f^{(2\alpha)}(x)
                        \right\vert^q+ \left\vert f^{(2\alpha)}(a)\right\vert^q\right)^\frac{1}{q} \nonumber\\
                    && \:\:\:\:\:\:\:\:\:\:\:\:\:\:\:\:\:\:\:\:\:\:\:\:\:\:\:\:\:\:
                        +(b-x)^{3\alpha}\left(\left\vert f^{(2\alpha)}(x)\right\vert^q
                             + \left\vert f^{(2\alpha)}(b)\right\vert^q\right)^\frac{1}{q}\Bigg]
\end{eqnarray}
which desired inequality \eqref{S}.
\end{proof}
\begin{corollary}
\label{corollary_2}
In Theorem \ref{Theorem-2}, if we choose $x =\frac{a+b}{2}$
and use the $s-$convexity of $|f^{(2\alpha)}|^q$, we get the following inequality:
\begin{eqnarray}
\label{t-c-1}
&&\left\vert\frac{1}{(b-a)^\alpha}\ _aI_b^\alpha f(t)-\frac{f(\frac{a+b}{2})}{\Gamma(1+\alpha)}
  \right\vert \nonumber \\
   &\leq&\left(\frac{\Gamma(1+2p\alpha)}{\Gamma(1+(2p+1)\alpha)}\right)^\frac{1}{p}
            \left(\frac{\Gamma(1+s\alpha)}{\Gamma(1+(s+1)\alpha)}\right)^\frac{1}{q}
                    \frac{\left(b-a\right)^{2\alpha}}{8^\alpha\Gamma(1+2\alpha)} \nonumber \\
                    &&\times \Bigg[\left(\left\vert f^{(2\alpha)}\left(\frac{a+b}{2}\right)
                        \right\vert^q+ \left\vert f^{(2\alpha)}(a)\right\vert^q\right)^\frac{1}{q}
                        +\left(\left\vert f^{(2\alpha)}\left(\frac{a+b}{2}\right)\right\vert^q
                             + \left\vert f^{(2\alpha)}(b)\right\vert^q\right)^\frac{1}{q}\Bigg] \nonumber\\
   &\leq&\left(\frac{\Gamma(1+2p\alpha)}{\Gamma(1+(2p+1)\alpha)}\right)^\frac{1}{p}
            \left(\frac{\Gamma(1+s\alpha)}{\Gamma(1+(s+1)\alpha)}\right)^\frac{1}{q}
                    \frac{\left(b-a\right)^{2\alpha}}{8^\alpha\Gamma(1+2\alpha)} \nonumber \\
                  &&  \times\Bigg[\left(\frac{\left(2^{s\alpha}+1\right)\left\vert f^{(2\alpha)}(a)
                        \right\vert^q+ \left\vert f^{(2\alpha)}(b)\right\vert^q}{2^{s\alpha}}\right)^\frac{1}{q}
                        +\left(\frac{\left\vert f^{(2\alpha)}(a)\right\vert^q+\left(2^{s\alpha}+1\right)\left\vert f^{(2\alpha)}(b)
                        \right\vert^q }{2^{s\alpha}}\right)^\frac{1}{q}
                        \Bigg] \nonumber \\
   &\leq&\left(\frac{\Gamma(1+2p\alpha)}{\Gamma(1+(2p+1)\alpha)}\right)^\frac{1}{p}
            \left(\frac{\Gamma(1+s\alpha)}{\Gamma(1+(s+1)\alpha)}\right)^\frac{1}{q}
                     \nonumber \\
                  && \times \frac{\left(1+\left(1+2^{s\alpha}\right)^\frac{1}{q}\right)\left(b-a\right)^{2\alpha}}
                    {2^{\left(3+\frac{s}{q}\right)\alpha}\Gamma(1+2\alpha)}\left(\left\vert f^{(2\alpha)}(a)\right\vert+\left\vert f^{(2\alpha)}(b)\right\vert\right).
\end{eqnarray}
While obtaining the last part of the inequality \eqref{t-c-1} it has been used the fact that
$\displaystyle\sum_{k=1}^{n}(u_k+v_k)^r \leq \displaystyle\sum_{k=1}^{n}(u_k)^r+\displaystyle\sum_{k=1}^{n}(v_k)^r,\,\,\,u_k,v_k\geq0,\,\,\,
1\leq k\leq n,\,\,0\leq r\leq 1$.
\end{corollary}
\begin{corollary}
By under assumptions of Theorem \ref{Theorem-2} and taking $\Theta:=\left\vert f^{(2\alpha)}(x)\right\vert_\infty$,
we have
\begin{eqnarray}
\label{Corollary-3}
&&\left\vert\frac{1}{(b-a)^\alpha}\ _aI_b^\alpha f(t)-\frac{f(x)}{\Gamma(1+\alpha)}
   +\frac{(2x-a-b)^{\alpha}f^{\alpha}(x)}{\Gamma(1+2\alpha)}\right\vert  \\
   &&\leq\left(\frac{\Gamma(1+2p\alpha)}{\Gamma(1+(2p+1)\alpha)}\right)^\frac{1}{p}
            \left(\frac{\Gamma(1+s\alpha)}{\Gamma(1+(s+1)\alpha)}\right)^\frac{1}{q}
                    \frac{3^\alpha\Theta}{\Gamma(1+2\alpha)}
                        \left[\frac{\left(b-a\right)^{2\alpha}}{12^\alpha}+\left(x-\frac{a+b}{2}\right)^{2\alpha}\right]. \nonumber
\end{eqnarray}
\end{corollary}
\begin{corollary}
If we take $x=\frac{a+b}{2}$ in Corollary \ref{Corollary-3},
we get
\begin{eqnarray}
\label{Corollary-3}
&&\left\vert\frac{1}{(b-a)^\alpha}\ _aI_b^\alpha f\left(\frac{a+b}{2}\right)-\frac{f(x)}{\Gamma(1+\alpha)}
   \right\vert  \\
   &&\leq\left(\frac{\Gamma(1+2p\alpha)}{\Gamma(1+(2p+1)\alpha)}\right)^\frac{1}{p}
            \left(\frac{\Gamma(1+s\alpha)}{\Gamma(1+(s+1)\alpha)}\right)^\frac{1}{q}
                    \frac{\Theta\left(b-a\right)^{2\alpha}}{4^\alpha\Gamma(1+2\alpha)} . \nonumber
\end{eqnarray}
\end{corollary}
\begin{theorem}
\label{Theorem-3}
Suppose that the assumptions of Lemma \ref{mean-lemma}  are satisfied.
If $|f^{(2\alpha)}|^q$  is generalized $s$-convex in the second sense where $s\in (0,1)$, then \label{eq1} for all $q\geq1$
\begin{eqnarray}
\label{H_S}
&&\left\vert\frac{1}{(b-a)^\alpha}\ _aI_b^\alpha f(t)-\frac{f(x)}{\Gamma(1+\alpha)}
   +\frac{(2x-a-b)^{\alpha}f^{(\alpha)}(x)}{\Gamma(1+2\alpha)}\right\vert \nonumber \\
   &\leq&\left(\frac{\Gamma(1+2\alpha)}{\Gamma(1+3\alpha)}\right)^{1-\frac{1}{q}}
            \frac{1}{\Gamma(1+2\alpha)(b-a)^\alpha}\nonumber\\
            &&\times\Bigg[\left(x-a\right)^{3\alpha}\Big(M(s,\alpha)\left\vert f^{(2\alpha)}(x)\right\vert^q
    +N(s,\alpha)\left\vert f^{(2\alpha)}(a)\right\vert^q\Big)^\frac{1}{q} \nonumber\\
   &&+\left(b-x\right)^{3\alpha}\Big(M(s,\alpha)\left\vert f^{(2\alpha)}(x)\right\vert^q
 +N(s,\alpha)\left\vert f^{(2\alpha)}(b)\right\vert^q\Big)^\frac{1}{q}\Bigg]
\end{eqnarray}
where $M(s,\alpha):\displaystyle=\frac{\Gamma(1+(s+2)\alpha)}{\Gamma(1+(s+3)\alpha)}$ \:\:\:\: and \:\:\:\:
$N(s,\alpha):\displaystyle=\frac{\Gamma(1+(s+2)\alpha)}{\Gamma(1+(s+3)\alpha)}
-2^\alpha\frac{\Gamma(1+(s+1)\alpha)}{\Gamma(1+(s+2)\alpha)}
+\frac{\Gamma(1+s\alpha)}{\Gamma(1+(s+1)\alpha)}$.
\end{theorem}
\begin{proof}
From Lemma \ref{mean-lemma} and by generalized power-mean inequality, we have
\begin{eqnarray}
\label{S-1}
   &&\left\vert\frac{1}{(b-a)^\alpha}\ _aI_b^\alpha f(t)-\frac{f(x)}{\Gamma(1+\alpha)}
   +\frac{(2x-a-b)^{\alpha}f^{(\alpha)}(x)}{\Gamma(1+2\alpha)}\right\vert \nonumber \\
   &\leq&\frac{(x-a)^{3\alpha}}{\Gamma(1+\alpha)\Gamma(1+2\alpha)(b-a)^{\alpha}}\bigg\{\int_{0}^{1}
   t^{2\alpha}\left\vert f^{(2\alpha)}(tx+(1-t)a)\right\vert(dt)^{\alpha} \nonumber \\
   &&+\frac{(b-x)^{3\alpha}}{\Gamma(1+\alpha)\Gamma(1+2\alpha)(b-a)^{\alpha}}\int_{0}^{1}
   t^{2\alpha}\Big|f^{(2\alpha)}(tx+(1-t)b)\Big|(dt)^{\alpha}\bigg\}  \nonumber  \\
   &\leq&\frac{(x-a)^{3\alpha}}{\Gamma(1+2\alpha)(b-a)^{\alpha}}
   \left(\frac{1}{\Gamma(1+\alpha)}\int_0^1t^{2\alpha}(dt)^\alpha\right)^{1-\frac{1}{q}} \nonumber \\
   &&\:\:\:\:\:\:\:\:\:\:\:\:\:\:\:\:\times\left(\frac{1}{\Gamma(1+\alpha)}\int_0^1t^{2\alpha}
   \left\vert f^{(2\alpha)}(tx+(1-t)a)\right\vert^q(dt)^\alpha\right)^\frac{1}{q} \nonumber \nonumber \\
   &&+\frac{(b-x)^{3\alpha}}{\Gamma(1+2\alpha)(b-a)^{\alpha}}
   \left(\frac{1}{\Gamma(1+\alpha)}\int_0^1t^{2\alpha}(dt)^\alpha\right)^{1-\frac{1}{q}} \nonumber \\
   &&\:\:\:\:\:\:\:\:\:\:\:\:\:\:\:\:\times\left(\frac{1}{\Gamma(1+\alpha)}\int_0^1t^{2\alpha}
   \left\vert f^{(2\alpha)}(tx+(1-t)b)\right\vert^q(dt)^\alpha\right)^\frac{1}{q}.
\end{eqnarray}
Since $\left\vert f^{(2\alpha)}\right\vert^q$ is generalized $s-$convex in the second sense,
we have
\begin{eqnarray}
\left\vert f^{(2\alpha)}(tx+(1-t)a)\right\vert^q
\leq t^{s\alpha}\left\vert f^{(2\alpha)}(x)\right\vert^q+(1-t)^{s\alpha}\left\vert f^{(2\alpha)}(a)\right\vert^q
\end{eqnarray}
and
\begin{eqnarray}
\left\vert f^{(2\alpha)}(tx+(1-t)b)\right\vert^q
\leq t^{s\alpha}\left\vert f^{(2\alpha)}(x)\right\vert^q+(1-t)^{s\alpha}\left\vert f^{(2\alpha)}(b)\right\vert^q.
\end{eqnarray}
Thus we can write,
\begin{eqnarray}
\label{H_S-1}
   &&\left\vert\frac{1}{(b-a)^\alpha}\ _aI_b^\alpha f(t)-\frac{f(x)}{\Gamma(1+\alpha)}
   +\frac{(2x-a-b)^{\alpha}f^{(\alpha)}(x)}{\Gamma(1+2\alpha)}\right\vert  \\
   &\leq&\frac{(x-a)^{3\alpha}}{\Gamma(1+2\alpha)(b-a)^{\alpha}}
   \left(\frac{1}{\Gamma(1+\alpha)}\int_0^1t^{2\alpha}(dt)^\alpha\right)^{1-\frac{1}{q}} \nonumber \\
   &&\:\:\:\:\:\:\:\:\:\:\:\:\:\:\:\:
   \times\left(\frac{1}{\Gamma(1+\alpha)}\int_0^1t^{2\alpha}
   \left[t^{s\alpha}\left\vert f^{(2\alpha)}(x)\right\vert^q+(1-t)^{s\alpha}
   \left\vert f^{(2\alpha)}(a)\right\vert^q\right](dt)^\alpha\right)^\frac{1}{q} \nonumber \\
   &&+\frac{(b-x)^{3\alpha}}{\Gamma(1+2\alpha)(b-a)^{\alpha}}
   \left(\frac{1}{\Gamma(1+\alpha)}\int_0^1t^{2\alpha}(dt)^\alpha\right)^{1-\frac{1}{q}} \nonumber \\
   &&\:\:\:\:\:\:\:\:\:\:\:\:\:\:\:\:\times\left(\frac{1}{\Gamma(1+\alpha)}
   \int_0^1t^{2\alpha}\left[t^{s\alpha}\left\vert f^{(2\alpha)}(x)\right\vert^q+(1-t)^{s\alpha}\left\vert f^{(2\alpha)}(b)\right\vert^q\right](dt)^\alpha\right)^\frac{1}{q} \nonumber.
\end{eqnarray}
Using Lemma \ref{belirli-integral}, we have
\begin{eqnarray}
\label{H-S-1}
\frac{1}{\Gamma(1+\alpha)}\int_0^1t^{2\alpha}(dt)^\alpha
=\frac{\Gamma(1+2\alpha)}{\Gamma(1+3\alpha)},
\end{eqnarray}
\begin{eqnarray}
\label{H-S-2}
\frac{1}{\Gamma(1+\alpha)}\int_0^1t^{2\alpha}t^{s\alpha}(dt)^\alpha
=\frac{\Gamma(1+(s+2)\alpha)}{\Gamma(1+(s+3)\alpha)}
\end{eqnarray}
and
\begin{eqnarray}
\label{H-S-3}
&&\frac{1}{\Gamma(1+\alpha)}\int_0^1t^{2\alpha}(1-t)^{s\alpha}(dt)^\alpha\nonumber \\
&&=\frac{\Gamma(1+(s+2)\alpha)}{\Gamma(1+(s+3)\alpha)}
-2^\alpha\frac{\Gamma(1+(s+1)\alpha)}{\Gamma(1+(s+2)\alpha)}
+\frac{\Gamma(1+s\alpha)}{\Gamma(1+(s+1)\alpha)}.
\end{eqnarray}
Substituting \eqref{H-S-1}, \eqref{H-S-2} and \eqref{H-S-3} in \eqref{H_S-1}, we get desired inequality \eqref{H_S}.
So proof of this theorem is complete.
\end{proof}
\begin{corollary}
\label{corollary_3}
In Theorem \ref{Theorem-3}, if we choose $x =\frac{a+b}{2}$
and use the $s-$convexity of $|f^{(2\alpha)}|^q$, we get the following inequality:
\begin{eqnarray}
\label{t-c-1}
&&\left\vert\frac{1}{(b-a)^\alpha}\ _aI_b^\alpha f(t)-\frac{f(\frac{a+b}{2})}{\Gamma(1+\alpha)}
  \right\vert
   \leq\left(\frac{\Gamma(1+2\alpha)}{\Gamma(1+3\alpha)}\right)^{1-\frac{1}{q}}
            \frac{\left(b-a\right)^{2\alpha}}{8^\alpha\Gamma(1+2\alpha)}\nonumber\\
            &&\times\Bigg[\left(M(s,\alpha)\left(\frac{\vert f^{(2\alpha)}(a)\vert+\vert f^{(2\alpha)}(b)\vert}{2^{s\alpha}}\right)
            +N(s,\alpha)\vert f^{(2\alpha)}(a)\vert\right)^\frac{1}{q}
            \nonumber\\
   &&+\left(M(s,\alpha)\left(\frac{\vert f^{(2\alpha)}(a)\vert+\vert f^{(2\alpha)}(b)\vert}{2^{s\alpha}}\right)
            +N(s,\alpha)\vert f^{(2\alpha)}(b)\vert\right)^\frac{1}{q}\Bigg] \nonumber \\
            &\leq&\left(\frac{\Gamma(1+2\alpha)}{\Gamma(1+3\alpha)}\right)^{1-\frac{1}{q}}
            \frac{\left(b-a\right)^{2\alpha}}{2^{\left(3+\frac{s}{q}\right)\alpha}\Gamma(1+2\alpha)}\nonumber\\
            &&\times\left(\left(M(s,\alpha)+2^{s\alpha}N(s,\alpha)\right)^\frac{1}{q}+M(s,\alpha)^\frac{1}{q}\right)
            \left(\left\vert f^{(2\alpha)}(a)\right\vert+\left\vert f^{(2\alpha)}(b)\right\vert\right) \nonumber\\
\end{eqnarray}
\end{corollary}
\begin{corollary}\label{Corollary-5}
By under assumptions of Theorem \ref{Theorem-3} and taking $\Theta:=\left\vert f^{(2\alpha)}(x)\right\vert_\infty$,
we have
\begin{eqnarray}
&&\left\vert\frac{1}{(b-a)^\alpha}\ _aI_b^\alpha f(t)-\frac{f\left(\frac{a+b}{2}\right)}{\Gamma(1+\alpha)}
   +\frac{(2x-a-b)^{\alpha}f^{\alpha}(x)}{\Gamma(1+2\alpha)}\right\vert  \\
   &&\leq\left(\frac{\Gamma(1+2\alpha)}{\Gamma(1+3\alpha)}\right)^{1-\frac{1}{q}}
                    \frac{3^\alpha\Theta\left(M(s,\alpha)+N(s,\alpha)\right)^\frac{1}{q}}{\Gamma(1+2\alpha)}
                    \left[\frac{\left(b-a\right)^{2\alpha}}{12^\alpha}+\left(x-\frac{a+b}{2}\right)^{2\alpha}\right] \nonumber .
\end{eqnarray}
\end{corollary}
\begin{corollary}
If we take $x=\frac{a+b}{2}$ in Corollary \ref{Corollary-5},
we get
\begin{eqnarray}
\label{Corollary-3}
&&\left\vert\frac{1}{(b-a)^\alpha}\ _aI_b^\alpha f\left(t\right)-\frac{f\left(\frac{a+b}{2}\right)}{\Gamma(1+\alpha)}
   \right\vert  \\
   &&\leq\left(\frac{\Gamma(1+2\alpha)}{\Gamma(1+3\alpha)}\right)^{1-\frac{1}{q}}
                    \frac{\Theta\left(M(s,\alpha)+N(s,\alpha)\right)^\frac{1}{q}\left(b-a\right)^{2\alpha}}{4^\alpha\Gamma(1+2\alpha)}
                   \nonumber .
\end{eqnarray}
\end{corollary}


\end{document}